\documentclass[11pt]{amsart}

\usepackage{amsthm,amsmath,amssymb,graphics, hyperref, graphicx,epstopdf,amsxtra, amscd, graphicx, verbatim}

\title[Domination and Projective Dimension]{Projective Dimension, Graph Domination Parameters, and Independence Complex Homology}
\author{Hailong Dao}
\author{Jay Schweig}
\address{Department of Mathematics, University of Kansas, 405 Snow Hall, Lawrence, KS 66045} 
\email{hdao@math.ku.edu}
\address{Department of Mathematics, University of Kansas, 405 Snow Hall, Lawrence, KS 66045} 
\email{jschweig@math.ku.edu}

\dedicatory{Dedicated to Craig Huneke on the occasion of his sixtieth birthday}
\keywords{Projective dimension, independence complex, graph domination, Hochster's formula, edge ideal}
\subjclass [2000]{13D02, 13P25, 05C10, 05C69, 05E45}

\newtheorem{theorem}{Theorem}[section]
\newtheorem{proposition}[theorem]{Proposition}
\newtheorem{corollary}[theorem]{Corollary}
\newtheorem{lemma}[theorem]{Lemma}
\newtheorem{definition}[theorem]{Definition}

\newtheorem{question}[theorem]{Question}
\newtheorem{example}[theorem]{Example}
\newtheorem{remark}[theorem]{Remark}
\newtheorem{observation}[theorem]{Observation}

\newcommand{\reg}{\text{reg}}   
\newcommand{\ep}{\epsilon}
\DeclareMathOperator{\st}{st}
\DeclareMathOperator{\ind}{ind}
\DeclareMathOperator{\bh}{BigHeight}
\DeclareMathOperator{\pd}{pd}
\DeclareMathOperator{\is}{Is}

\DeclareMathOperator{\height}{height}

\begin{document} 

\maketitle

\begin{abstract}
We construct several pairwise-incomparable bounds on the projective dimensions of edge ideals.  Our bounds use combinatorial properties of the associated graphs; in particular we draw heavily from the topic of dominating sets.  Through Hochster's Formula, we recover and strengthen existing results on the homological connectivity of graph independence complexes.
\end{abstract}

\section{Introduction}

Let $G$ be a graph with independence complex $\ind(G)$ and fix a ground field $\bf{k}$.  A much-studied question in combinatorial and algebraic graph theory is the following:

\begin{question}\label{mainq}
What are non-trivial bounds on the biggest integer $n$ such that $\tilde H_i(\ind(G), {\bf k}) =0$ for $0\leq i \leq n$? 
\end{question}

Answers to the above question immediately give constraints on the homotopy type of the independence complex. They can also be applied to various  other problems, such as Hall type theorems (see \cite{ah}).  Thus, the question has drawn attention from many researchers (see, for instance, \cite{abz}, \cite{barmak}, \cite{DE}, or \cite{mesh}, and references given therein). Usually, the tools in such work come from combinatorial topology.  

In our work, we take a different approach; through the well-known Hochster's Formula (see Theorem \ref{Hochster}), we relate Question \ref{mainq} to the projective dimension of the associated edge ideal.  In particular, we begin with the following observation.

\begin{observation}
Any upper bound for the projective dimension of a graph's edge ideal provides a lower bound for the first non-zero homology group of the graph's independence complex.  
\end{observation}

Thus, our work revolves around combinatorially constructed bounds for the projective dimension of edge ideals.  The bounds we obtain on independence complex homology typically recover or improve on what is known in the literature for general graphs as well as several well-studied subclasses: for example chordal, generalized claw-free, and finite subgraphs of integer lattices in any dimension (see Section \ref{homsec}). Our proofs are sometimes subtle but quite elementary, and follow an axiomatized inductive approach.  Thus, our main methods make the problem of bounding the projective dimensions or homological connectivity of the independence complex of all graphs in a given class a rather mechanical task (see Section \ref{metas}, especially Theorems \ref{ME} and \ref{secondmeta}).  Our only required tools are a special case of Hochster's Formula (Corollary \ref{hochster}) and a straightforward short exact sequence (Lemma \ref{pdinduct}).



Most of our bounds make use of various \emph{graph domination parameters}.  Dominating sets in graphs have received much attention from those working on questions in combinatorial algorithms, optimization, and computer networks (see \cite{dombook}).  We believe that our results are the first to systematically relate domination parameters of a graph to its edge ideal's projective dimension (although the connection between domination parameters and independence homology has been previously explored; see \cite{mesh}).  
        
In particular, our study leads to the definition of \emph{edgewise-domination}, a new graph domination parameter which works especially well with bounding projective dimension (see the beginning of Section 4 for definition and Theorem \ref{edgedom}). 


Our paper is organized as follows.  We start by reviewing the necessary background in both commutative algebra and graph theory.  Section \ref{metas} is concerned with the technical background and several key theorems we use in later sections.  It also contains our first upper bounds for the projective dimension of an edge ideal.  These bounds use invariants such as the chromatic number of a graph's complement and the maximum degree of an edge.  In Section \ref{dominationsection}, we relate the projective dimension of a graph's edge ideal to various domination parameters of the graph, and introduce a new domination parameter. In Section \ref{chordalsection} we apply our methods to situations when an exact formula for the projective dimension can be found.  In particular, we recover a known formula for the projective dimension of a chordal graph.  We also introduce a new class of graphs we call long graphs which behave nicely with respect to dominating sets (see \ref{long}).  Section \ref{homsec} is concerned with bounds on the homology of graph independence complexes.  Using Hochster's Formula as a bridge, we manage to recover and/or strengthen many results on the connectivity of such complexes.  We conclude with some examples and a discussion of further research directions in Section \ref{lastsec}. 

\section{Preliminaries and Background}\label{backgroundsection}

\subsection{Graph Theory}

Most of our graph theory terminology is fairly standard (see \cite{bollobas}).  All our graphs are finite and \emph{simple} (meaning they have no loops or parallel edges). 

For a graph $G$, let $V(G)$ denote its vertex set.  We write $(v, w)$ to denote an edge of $G$ with endpoints $v$ and $w$ (all our graphs are \emph{un}directed, so the order of $v$ and $w$ is immaterial).  If $v$ is a vertex of $G$, we let $N(v)$ denote the set of its neighbors.  If $X$ is a subset of vertices of $G$, we also set $N(X) = \bigcup_{v \in X} N(x)$. 

If $G$ is a graph and $W \subseteq V(G)$, the \emph{induced subgraph} $G[W]$ is the subgraph of $G$ with vertex set $W$, where $(v, w)$ is an edge of $G[W]$ if and only if it is an edge of $G$ and $v, w \in W$.  If $v\in V(G)$, the \emph{star} of $v$, $\st(v)$, is the induced subgraph $G[N(v) \cup \{v\}]$.  We also write $G^c$ for the \emph{complement} of $G$, the graph on the same vertex set as $G$ where $(v,w)$ is an edge of $G^c$ whenever it is not an edge of $G$.  

We also write $\is(G)$ to denote the set of isolated vertices of $G$, and we let $\overline{G} = G - \is(G)$.  

We also write $K_{m, n}$ to denote the complete bipartite graph with $m$ vertices on one side and $n$ on the other.  Recall that $K_{1, 3}$ is known as a \emph{claw}, and graphs with no induced subgraph isomorphic to $K_{1,  3}$ are called \emph{claw-free}.  

Most of our proofs use (sometimes nested) induction; Thus we are interested in classes of graphs closed under deletion of vertices:
\begin{definition}
Let $\mathcal{C}$ be a class of graphs such that $G - x$ is in $\mathcal{C}$ whenever $G \in \mathcal{C}$ and $x$ is a vertex of $G$.  We call such a class \emph{hereditary}.  
\end{definition}

Note that, by definition, hereditary classes of graphs are closed under the removal of induced subgraphs (such as stars of vertices).

Most widely-studied classes of graphs arising in graph theory are hereditary (such as claw-free graphs, perfect graphs, planar graphs, graphs not having a fixed graph $G$ as a minor, et cetera).  

\subsection{Algebraic Background}

Fix a field $\bf{k}$, and let $S = {\bf k}  [x_1, x_2, \ldots, x_n]$ (as ${\bf k}$ is fixed throughout, we suppress it from the notation).  If $G$ is a graph with vertex set $V(G) = \{x_1, x_2, \ldots, x_n\}$, the \emph{edge ideal} of $G$ is the monomial ideal $I(G)\subseteq S$ given by 
\[
I(G) = (x_ix_j: (x_i, x_j) \text{ is an edge of } G).
\]

Edge ideals have been heavily studied (see \cite{DE}, \cite{fv}, \cite{hv}, \cite{mv}, and references given therein).  We say a subset $W \subseteq V(G)$ is \emph{independent} if no two vertices in $W$ are adjacent (equivalently, $G[W]$ has no edges).  Closely related to the edge ideal $I(G)$ of $G$ is its \emph{independence complex}, $\ind(G)$, which is the simplicial complex on vertex set $V(G)$ whose faces are the independent sets of $G$.  Note that $I(G)$ is the \emph{Stanley-Reisner ring} of $\ind(G)$.  

For a graph $G$, we write $\pd(G)$ and $\reg(G)$ as shorthand for $\pd(S/I(G))$ and $\reg(S/I(G))$, respectively.  

Central to the link between commutative algebra and combinatorics is \emph{Hochster's Formula}, which relates the Betti numbers of an ideal to its Stanley-Reisner complex (see, for instance, \cite{millersturmfels}).  In our case, we have the following.

\begin{theorem}[Hochster's Formula]\label{Hochster}
Let $\Delta$ be the Stanley-Reisner complex of a squarefree monomial ideal $I \subseteq S$.  For any multigraded Betti number $\beta_{i, m}$ where $m$ is a squarefree monomial of degree $\geq i$, we have 
\[
\beta_{i-1, m}(I) = \dim_{{\bf k}} (\tilde{H}_{\deg m - i - 1}(\Delta[m]), {\bf k}),
\]
where $\Delta[m]$ is the subcomplex of $\Delta$ consisting of those faces whose vertices correspond to variables occurring in $m$. 
\end{theorem}

In particular, we are interested in the following specialization of Hochster's Formula to simple graphs and their independence complexes.  Here and throughout, if $\Delta$ is a complex, we write $\tilde{H}_k(\Delta) = 0$ to mean that the associated homology group has rank zero.  

\begin{corollary}\label{hochster}
Let $G$ be a graph with vertex set $V$.  Then $\pd(G)$ is the least integer $i$ such that 
\[
\tilde{H}_{|W| - i - j -1}(\ind(G[W])) = 0
\]
for all $j > 0$ and $W \subseteq V$.  Moreover, $\reg(G)$ is the greatest value of $k$ so that
\[
\tilde{H}_{k - 1}(\ind(G[W])) \neq 0
\]
for some subset $W\subseteq V$.  
\end{corollary}

\section{General Formulas for Bounding Projective Dimension}\label{metas}
This section contains the  key technical process which we shall utilize for the rest of the paper. 
Our first  result  provides a general framework for bounding the projective dimension of a graph's edge ideal. 

\begin{theorem}\label{ME}
Let $\mathcal{C}$ be a hereditary class of graphs and let $f: \mathcal{C} \rightarrow \mathbb{R}$ be a  function satisfying the following conditions:  
\begin{enumerate}
\item $f(G) \leq |V(G)|$ when $G$ is a collection of isolated vertices. 

Furthermore, for any $G \in \mathcal{C}$ with at least an edge there exists a nonempty set of vertices $v_1, v_2, \ldots, v_k$ such that if we set $G_i = G - v_1 - v_2 - \ldots -v_i$ for $ 0 \leq i \leq k$ ($G_0 =G$), then:
\item $f(G_i- \st_{G_i}  v_{i+1}) +1 \geq f(G)$ for $ 0 \leq i \leq k-1$.
\item $f(\overline{G_k}) + |\is(G_k)|  \geq f(G)$.
\end{enumerate}  

Then for any graph $G \in \mathcal C$
\[
\pd(G) \leq |V(G)| - f(G).
\]
\end{theorem}

Before proving Theorem \ref{ME}, we need the following helpful lemma, also used in \cite{craigandus}, which will prove invaluable in our work.  

\begin{lemma}[see, for instance, \cite{craigandus}]\label{pdinduct}

Let $x$ be a vertex of a graph $G$.  Then
\[
\pd(G) \leq \max \{ \pd(G - \st x) + \deg x, \pd(G - x) +1\}.
\]
\end{lemma}

\begin{proof}
Consider the following exact sequence:
\[
0 \rightarrow S/(I(G) \colon x) \rightarrow S/I(G) \rightarrow S/(I(G), x) \rightarrow 0.
\]
This gives us that $\pd(I(G)) \leq \max \{ \pd(I(G) \colon x), \pd(I(G), x)\}$.  It is easily seen that $(I(G) \colon x)$ is the ideal generated by $I(G - \st x)$ along with the variables in $N(x)$, whereas $(I(G), x)$ is the ideal generated by $I(G -x)$ and $x$.  Thus, $\pd(I(G) \colon x ) = \pd(I(G - \st x)) +  \deg x$ and $\pd(I(G), x) = \pd(I(G - x)) + 1$.  
\end{proof}

\begin{proof}[Proof of Theorem \ref{ME}]
We will argue by contradiction. Suppose there is a counterexample $G$ with a minimal number of vertices. Let $v_1, v_2, \ldots, v_k$ be a set of vertices satisfying conditions (2) and (3). We will prove by induction on $i$ that $$\pd(G_i) \leq \pd(G_{i+1}) +1 \  \  (*) $$
for $0\leq i\leq n-1$. 

We start with $i=0$. By Lemma \ref{pdinduct} we need to show that $\pd(G) > \pd(G - \st v_1) + \deg v_1$. If this fails, then
\begin{align*}
\pd(G) &\leq  \pd(G_1 - \st v_1) +\deg v_1 \\ 
& \leq |V(G_1 - \st v_1)| - f(G_1 - \st v_1) +\deg v_1 \\
& =  |V(G)| -\deg v_1-1 - f(G_1 - \st v_1) + \deg v_1\\
& \leq |V(G)| -f(G).
\end{align*}
Now, suppose we have proved (*) for $0 \leq i \leq j-1$. We argue just like above for the induction step. By Lemma \ref{pdinduct} we need to show that $\pd(G) > \pd(G_j - \st_{G_j} v_{j+1}) + \deg_{G_j} v_{j+1} + j$. If this is not true then,
\begin{align*}
\pd(G)  & \leq  \pd(G_{j} - \st_{G_j}v_{j+1}) + \deg_{G_j} v_{j+1} +j \\
   &\leq |V(G_{j} - \st_{G_j} v_{j+1})| - f(G_j - \st_{G_j} v_{j+1}) +\deg_{G_j} v_{j+1} +j  \\
   &=  |V(G)| - j - \deg v_1-1 - f(G_j - \st_{G_j}v_{j+1}) + \deg v_{j+1} +j \\
   & \leq |V(G)| -f(G).
   \end{align*}
which contradicts our choice of $G$. Now that (*) is established, it follows that $\pd(G) \leq \pd(G_k) +k$. But by condition (3): 
\begin{align*}
\pd(G_k) + k   & = \pd (\overline{G_k}) + k\\
& \leq |V(\overline{G_k})| - f(\overline{G_k}) + k \\
& =  |V(G)| - k -|\is(G_k)| - f(\overline{G_k}) + k\\
& \leq |V(G)| -f(G)
\end{align*}
which again contradicts our choice of  $G$, proving the theorem.
\end{proof}

The following consequence of Theorem \ref{ME} will also be rather helpful.

\begin{theorem}\label{secondmeta}
Let $\mathcal{C}$ be a hereditary class of graphs and let $h: \mathcal{C} \rightarrow \mathbb{R}$ be a function satisfying the following conditions.
\begin{enumerate}
\item $h$ is non-decreasing.  That is, $h(G - v) \leq h(G)$ for any vertex $v$ of $G$. 

\item For any $G \in \mathcal{C}$ there exists a vertex $v$ so that the neighbors of $v$ admit an ordering $v_1, v_2, \ldots, v_k$ satisfying the following property: if $G_i = G - v_1 - v_2 - \ldots - v_i$, then $i + d_{i+1} + 1 \leq h(G)$ for all $i < k $ (where $d_{i+1}$ denotes the degree of $v_{i+1}$ in $G_i$).  Furthermore, $h(G) \geq k+1$.  
 
\end{enumerate}
Then 
\[
\pd(G) \leq n\left(1 - \frac{1}{h(G)}\right).
\]
\end{theorem}

\begin{proof}
Define a function $f:\mathcal{C} \rightarrow \mathbb{R}$ by $f(G) = \frac{|V(G)|}{h(G)}$.  We claim that $f$ satisfies the conditions of Theorem \ref{ME}.  Condition 1 is immediately verified.  For condition 2, fix $i$, and suppose $G_i - \st_{G_i} v_{i+1}$ has no isolated vertices.  Then 
\begin{align*}
f(G_i - \st_{G_i} v_{i+1}) + 1 &= \frac{|V(G)| - i - d_{i+1}-1}{h(G_i - \st_{G_i} v_{i+1})} + 1\\
& \geq \frac{|V(G)| - i - d_{i+1}-1}{h(G)} + 1 \\
& = \frac{|V(G)| - i - d_{i + 1} - 1 + h(G)}{h(G)} \\
& \geq \frac{|V(G)|}{h(G)} = f(G).
\end{align*}

Finally, for property 3, note that $G_k$ has at least one isolated vertex (namely, $v$), meaning $|\is(G_k)| \geq 1$.  Thus, we have
\begin{align*}
f(\overline{G_k}) + |\is(G_k)| & \geq \frac{|V(G)| - k - |\is(G_k)|}{h(\overline{G_k})} + \is(G_k)\\
& \geq \frac{|V(G)| - k - |\is(G_k)| + h(G)|\is(G_k)|}{h(G)} \\
& \geq \frac{|V(G)| - k  - |\is(G_k)| + (k + 1) |\is(G_k)|}{h(G)} \\
& \geq \frac{|V(G)| - k -|\is(G_k)| + k + |\is(G_k)|}{h(G)} = \frac{|V(G)|}{h(G)} = f(G),
\end{align*}
completing the proof.
\end{proof}

\begin{definition}
For a graph $G$ and a scalar $\alpha >0$, call an edge $(x, y)$ of $G$ \emph{$\alpha$-max} if it maximizes the quantity $\deg(x) + \alpha \deg(y)$. 
\end{definition}

Using Theorem \ref{secondmeta}, we can provide a bound for the projective dimension of a graph that has no induced $K_{1, m+1}$ as a subgraph.  First, we need the following lemma.

\begin{lemma}\label{propm}
Let $G$ be a graph on $n$ vertices such that $G[W]$ has at least one edge for any $W \subseteq V(G)$ with $|W| = m$ (or, equivalently, $\dim(\ind(G)) <  m - 1$).  Then $G$ contains a vertex of degree at least $\frac{n}{m-1} - 1$.  
\end{lemma}

\begin{proof}
We induct on $m$.  If $m = 2$ then $G$ is complete, and the bound holds.  Now suppose $m > 2$, and let $v$ be a vertex of $G$ of maximal degree.  Let $G'$ be $G$ with $v$ and its neighbors removed.  Then $G'[W]$ must have an edge for any $W\subseteq V(G')$ with $|W| = m-1$, since $G[W \cup \{v\}]$ must contain an edge, and this edge cannot have $v$ as an endpoint.  By induction, a maximal degree vertex of $G'$ has degree at least $\frac{|V(G')|}{(m-1) - 1} - 1 = \frac{n - \deg v - 1}{m - 2} - 1$.  Since this degree cannot exceed $\deg v$, we have 
\[
\frac{n - \deg v - 1}{m -2} - 1 \leq \deg v \Rightarrow n - m + 1 \leq (m-1) \deg v \Rightarrow \frac{n}{m-1} - 1 \leq \deg v.
\]
\end{proof}

\begin{theorem}\label{mcondition}
Suppose $G$ is a graph on $n$ vertices containing no induced $K_{1, m+1}$, and let $(x, y)$ be an $\left(\frac{m-1}{m}\right)$-max edge, with $d  = \deg(x) \geq \deg(y) = e$.  Then 
\[
\pd(G) \leq n \left(1 - \frac{1}{d+\frac{m-1}{m}e +1} \right).
\]
\end{theorem}

\begin{proof}
We use Theorem \ref{secondmeta} with the function $h(G) = d + \frac{m-1}{m}e +1$ where $(x, y)$ is an $\left(\frac{m-1}{m}\right)$-max edge of $G$ with $\deg x = d$ and $\deg y = e$.  The function $h$ is easily seen to be decreasing.  Now let $\{v_1, v_2, \ldots, v_e\}$ be the neighbors of $y$.  We reorder the neighbors of $y$ as follows.  Let $v_e$ be a vertex of maximal degree in the induced subgraph $G[\{v_1, v_2, \ldots, v_e\}]$, and in general let $v_i$ be a vertex of maximal degree in the subgraph $G[\{v_1, v_2, \ldots, v_i\}]$.  

As in Theorem \ref{secondmeta}, let $d_{i+1}$ denote the degree of $v_{i+1}$ in the induced subgraph $G_i = G - v_1 - v_2 - \cdots - v_i$.  We need to show that $h(G) \geq i + d_{i+1} + 1$ for all $i < e$.  Fix some $i$, let $G' = G[\{v_1, v_2, \ldots, v_{i+1}\}]$, and let $S \subseteq \{v_1, v_2, \ldots, v_i\}$ be the set of non-neighbors of $v_{i+1}$.  Writing $\delta$ for the degree of $v_{i+1}$ in $G'$, note that $|S| = i - \delta$.  Because $G$ is $K_{1, m+1}$-free, $S$ cannot have an independent set of size $m$; if it did, these $m$ vertices, together with $y$ and $v_{i+1}$, would form an induced $K_{1, m+1}$.  Thus, by Lemma \ref{propm}, some vertex of $S$ must have degree (in $G'$) at least $|S|/(m-1) - 1$.  Because $v_{i+1}$ was chosen as a maximal degree vertex of $G'$, we have $\delta \geq |S|/(m-1) -1 = (i - \delta)/(m-1) -1 \Rightarrow \delta\geq  (i - m +1)/m$.  Using the fact that $i \leq e-1$, we have 
\begin{align*}
i + d_{i+1} + 1 &= i + \deg(v_{i+1}) - \delta + 1 \\
& \leq i + \deg(v_{i+1}) - \frac{i - m + 1}{m} + 1 \\
& = \deg(v_{i+1}) + \frac{m-1}{m}i +\frac{m-1}{m} +1 \\
& \leq \deg(v_{i+1}) + \frac{m-1}{m}e +1.
\end{align*}
Finally, because $(v_{i+1}, y)$ is an edge of $G$, the above quantity is $\leq h(G)$.  It remains to be shown that $h(G) \geq e+1$, but this is immediate.  
\end{proof}

Recall that, for $\ell \geq 1$, the $\mathbb{Z}^\ell$ \emph{lattice} is the infinite graph whose vertices are points in $\mathbb{R}^\ell$ with integer coordinates, where two vertices are connected by an edge whenever they are a unit distance apart. 

\begin{theorem}\label{zell}
Let $G$ be a  subgraph of the $\mathbb{Z}^\ell$ lattice with $n$ vertices. Then $\pd(G) \leq n(1- \frac{1}{2\ell + 1})$. 
\end{theorem}

\begin{proof}
We view each $v \in V(G)$ as an $\ell$-tuple in $\mathbb{Z}^d$, and write $v^i$ to denote its $i^\text{th}$ coordinate.  Without loss, we may assume that $\min\{v^1: v\in V\} = 0$ (otherwise, we can simply translate $G$ so that this is true).  Similarly, restricting to those $v\in V$ with first coordinate zero, we may also assume that $\min\{ v^2 : v\in V, v^1 = 0\} = 0$.  In general, we can assume that $\min\{v^i : v\in V, v^1 = v^2 = \cdots = v^{i-1} = 0\} = 0$ for all $i$.  Thus, $G$ contains no vertex whose first non-zero coordinate is negative.  

Now let $v$ denote the origin (which, given the above assumptions, must be a vertex of $G$).  For any $i$ let $v_i$ denote the vertex with $v_i^i = 1$ and $v_i^j = 0$ for $j \neq i$.  Note that the set of neighbors of $v$ is contained in $\{v_1, v_2, \ldots, v_\ell\}$.  Now fix $i$, and let $w \in V$ be a neighbor of $v_i$.  Then only one coordinate of $w$ can differ from $v_i$.  If this coordinate is $w^j$ for $j < i$, then we can only have $w^j = 1$, since $G$ contains no vertex whose first non-zero coordinate is negative.  If $w$ differs from $v_i$ in the $j^\text{th}$ component for $j \geq i$, then either $w^j = v^j - 1$ or $w^j = v^j +1$.  Thus, $\deg(v_i) \leq i - 1 + 2(\ell - i + 1)$.

We apply Theorem \ref{secondmeta} with $h(G) = 2\ell +1$.  List the neighbors of $v: v_{i_1}, v_{i_2}, \ldots, v_{i_k}$, where $i_1 < i_2 < \cdots < i_k$.  Then, for all $j$, we have
\[
j + \deg(v_{i_{j+1}}) + 1 \leq i_j + (i_j + 2(\ell - i_j)) + 1 = 2\ell + 1 = h(G).
\]
The function $h(G)$ is easily seen to satisfy the other requirements of Theorem \ref{secondmeta}.
\end{proof}

\begin{corollary}\label{inddim}
Suppose $G$ is a graph on $n$ vertices.  If the dimension of $\ind(G)$ is $< m$, then $\pd(G)$ satisfies the bound of Theorem \ref{mcondition}.
\end{corollary} 

\begin{proof} 
Since $\ind(G)$ cannot contain an $m$-dimensional face, the induced graph $G[W]$ must contain an edge for every subset of vertices $W$ with $|W| \geq m+1$.  Thus, $G$ cannot contain an induced $K_{1, m+1}$, and so Theorem \ref{mcondition} applies.  
\end{proof}

\begin{corollary}\label{chromatic}
Let $G$ be a graph on $n$ vertices.  Then 
\[
\pd(G) \leq n \left( 1 - \frac{1}{d+\frac{\chi(G^c) - 1}{\chi(G^c)} e + 1} \right),
\] 
where $\chi(G^c)$ is the chromatic number of $G^c$ and $d = \deg(x) \geq \deg(y) = e$ for some $\left(\frac{\chi(G^c) - 1}{\chi(G^c)}\right)$-max edge $(x, y)$. 
\end{corollary} 

\begin{proof}
Let $\chi(G^c) = m$.  Then $G^c[W]$ cannot be complete for any $W \subseteq V$ with $|W| = m+1$, meaning $G[W]$ has at least one edge for any such $W$, and we can apply Corollary \ref{inddim}. 
\end{proof}

\begin{example}
{\rm As an example, let $G = K_{d, d}$.  Then $G^c$ is the disjoint union of two copies of $K_d$, meaning $\chi(G^c) = d$, and every edge is $\left(\frac{d-1}{d}\right)$-max.  Using Corollary \ref{chromatic}, we obtain
\[
\pd(G) \leq 2d \left( 1 - \frac{1}{d+ \frac{d-1}{d}d + 1}\right) = 2d-1,
\]
meaning our bound is sharp in this case (see \cite{manoj}).}
\end{example}

\begin{example}
{\rm Let $G = K_n$.  Then $d = n-1$, and $G^c$ is a collection of isolated vertices, meaning $\chi(G^c) =1 $.  Using Corollary \ref{chromatic} again, we have
\[
\pd(G) \leq n \left( 1 - \frac{1}{(n-1)+ \frac{1-1}{1}(n-1) + 1}\right) = n-1,
\]
so our bound is sharp in this case as well (see \cite{manoj}). }
\end{example}

It is interesting to note that the chromatic number of the complement of a graph is also related to the regularity of the graph via the following easy observation.

\begin{observation}\label{regchrom}
For any graph $G$, we have $\reg(G) \leq \chi(G^c)$.
\end{observation}

\begin{proof}
The maximum cardinality of an independent set in $G$ is equal to the maximum cardinality of a clique in $G^c$, which is less than or equal to $\chi(G^c)$.  Thus no induced subcomplex of $\ind(G)$ can have homology in a dimension higher than $\chi(G^c)-1$.  The observation then follows from Hochster's formula (Theorem \ref{hochster}).
\end{proof}

\section{Domination Parameters and Projective Dimension}\label{dominationsection}

A central theme in this paper is the relationship between the projective dimension of a graph's edge ideal and various graph domination parameters.  The subject of domination in graphs has been well-studied, and is of special interest to those working in computer science and combinatorial algorithms. 

We first recall a catalog of basic domination parameters.  Let $G$ be a graph, and recall that a subset $A \subseteq V(G)$ is \emph{dominating} if every vertex of $V(G) \setminus A$ is a neighbor of some vertex in $A$ (that is, $N(A) \cup A = V(G)$.  
\begin{enumerate}
\item $\gamma(G) = \min \{|A| : A \subseteq V(G)$ is a dominating set of $G\}.$\\
\item $i(G) =  \min \{|A| : A \subseteq V(G)$ is independent and a dominating set of $G\}.$ \\

If $G$ is empty, we set $\gamma(G) = i(G) =0$.  For any subset $X \subseteq V(G)$, we let $\gamma_0(X, G)$ denote the minimal cardinality of a subset $A \subseteq V(G)$ such that $X \subseteq N(A)$ (note that we allow $X \cap A \neq \emptyset$).\\

\item $\gamma_0(G) = \gamma_0(V(G), G)$.  That is, $\gamma_0(G)$ is the least cardinality of a subset $A \subseteq V(G)$ such that every vertex of $G$ is adjacent to some $a \in A$. \\

\item $\tau(G) = \max \{ \gamma_0(A, G): A\subseteq V(G)$ is independent$\}$. \\

We also introduce a new graph domination parameter, which we call \emph{edgewise domination}.  Note that this differs from the existing notion of \emph{edge-domination}, which is often ignored in the literature, as it is equivalent to domination in the associated line graph.  \\

\item If $E$ is the set of edges of $G$, we say a subset $F$ of $E$ is \emph{edgewise dominant} if any $v \in G$ is adjacent to an endpoint of some edge $e \in F$.  We define 
\[
\epsilon(G) = \min\{|F| : F\subseteq E \ \ {\text{is edgewise dominant}} \}.
\] 
\end{enumerate}

The following proposition compares the domination parameters.  Part of this proposition is likely well-known, but we include a proof as it seems to lack a convenient reference.  See also Example \ref{eptau}. 

\begin{proposition}\label{gammaind}
For any $G$, $i(G) \leq \gamma(G)$ and  $\tau(G) \leq \gamma(G)$. Furthermore $\epsilon(G) \geq \frac{\gamma_0(G)}{2}$.
\end{proposition}

\begin{proof}
The first inequality is obvious.  Let $X\subseteq V$ be a dominating set of $G$ of minimal cardinality, and let $A\subseteq V$ be an independent set with $\tau(G) = \gamma_0(A , G)$.  Then $A \subseteq (N(X) \cup X)$, by definition.  If $x\in A \cap X$, then $N(x) \cap A = \emptyset$ (otherwise $A$ would not be independent).  For each $x \in X \cap A$, replace $x$ with one if its neighbors, and call the resulting set $X'$.  Then $A \subseteq N(X')$.  Since $|X| = |X'| \geq \gamma_0(A, G)$, we have $\gamma(G) = |X| = |X'| \geq \tau(G)$. 

We now prove the last inequality. Indeed, let $E$ be the set of edges of $G$, and let $X \subseteq E$ be an edgewise-dominant set of $G$.  If we let $A$ be the set of vertices in edges of $E$, then $A$ is easily seen to be strongly dominant, meaning $|E| \geq \frac{|A|}{2} \geq \frac{\gamma_0(G)}{2}$.
\end{proof}

In this section we prove several bounds for the projective dimension of an arbitrary graph.  Here we state the amalgam of the results that follow.  

\begin{corollary}\label{amalgam}
Let $G$ be a graph on $n$ vertices (and without isolated vertices).  Then
\[
n - i(G) \leq \pd(G) \leq n - \max\{ \epsilon(G), \tau(G)\}.
\]
\end{corollary}

Using the tools developed in Section \ref{metas}, we now turn to bounding the projective dimensions of graphs via various domination parameters.  Our first two results of this section, Theorems \ref{edgedom} and \ref{pdind}, hold for all graphs.

\begin{theorem}\label{edgedom}
Let $G$ be a graph on $n$ vertices.  Then
\[
\pd(G) \leq n - \epsilon(G).
\]
\end{theorem}

\begin{proof}
We claim that the function $\ep(G)$ satisfies the conditions of Theorem \ref{ME}. Let $G$ be any graph with at least one edge. Pick any vertex $v$ such that $ N(v) = \{v_1,\cdots, v_n \} $ is non-empty.  For any $i$, if we let $E_i$ be an edgewise-dominant set of $G_i -\st_{G_i}v_{i+1}$ then clearly $E_i \cup (v , v_i)$ dominates $G$. So 
$$\ep(G_i -\st_{G_i}v_{i+1}) +1 \geq \ep(G) $$
which implies condition (2). By the same reasoning $\ep(G_n)+1 \geq \ep(G)$. But $v \in \is(G_n)$, so $\is(G_n) \geq 1$, and thus condition (3) is also satisfied.
\end{proof}

We can also prove a similar bound involving $\tau(G)$.  While most of our proofs bounding projective dimension use Theorem \ref{ME} or Theorem \ref{secondmeta}, the following theorem is best proved using only Lemma \ref{pdinduct}.

\begin{theorem}\label{pdind}
Let $G$ be a graph without isolated vertices.  Then
\[
\pd(G) \leq n -\tau(G),
\]
where $n$ is the number of vertices of $G$. 
\end{theorem}

\begin{proof}
Let $A \subseteq V(G)$ be an independent set witnessing $\tau(G)$.  That is, $\gamma_0(A, G) = \tau(G)$.  Let $X \subseteq V(G)$ be such that $A \subseteq N(X)$ and $|X| = \gamma_0(A, G)$.  We use the two cases of Lemma \ref{pdinduct}.  Pick $v \in X$.  Then $v \notin A$, since $A$ is independent.  First, suppose that $\pd(G) \leq \pd(\overline{G - \st v}) + \deg v$, and let $B \subseteq A$ be the vertices of $A$ that are isolated by the removal of $\st v$ from $G$.  Let $Y$ be a set of vertices of $\overline{G- \st v}$ realizing $\gamma_0(A - B, \overline{G - \st v})$ (that is, $A - B \subseteq N(Y)$ and $|Y| = \gamma_0(A - B, \overline{G - \st v})$).  Now choose a neighbor in $G$ of each $b \in B$, and let $Z$ be the set of all these neighbors (so $|Z| \leq |B|$).  Then $A \subseteq N(Y \cup Z \cup v)$, so 
\[
\gamma_0(A - B, \overline{G - \st v}) = |Y| \geq \gamma_0(A, G) - |Z| - 1 \geq \gamma_0(A, G) - |B| - 1.
\]
Thus, by induction on the number of vertices of $G$ (the base case with two vertices and one edge being trivial), we have
\begin{align*}
\pd(G) &= \pd(G - \st v) + \deg v \\ 
&\leq (n - |B| - 1) - \tau(\overline{G - \st v}) \\ 
&\leq (n - |B| - 1) - \gamma_0(A - B, \overline{G - \st v}) \\
&\leq (n - |B| - 1) - (\gamma_0(A, G) - |B| - 1) = n - \tau(G).  
\end{align*}

Thus, we can assume that $\pd(G) \leq \pd(G - v) + 1$.  We examine two subcases.  First, suppose that no vertices of $A$ are isolated by the removal of $v$ from $G$ (that is, $A$ is a subset of the vertices of $\overline{G -v}$).  Then any subset $Y$ of vertices of $\overline{G - v}$ which strongly dominates  $A$ in $\overline{G - v}$ also strongly dominates $A$ in $G$, so $\tau(G) = \gamma_0(A, G) \leq \gamma_0(A, \overline{G - v}) \leq \tau(\overline{G - v})$, and we have $\pd(G) \leq \pd( \overline{G - v}) + 1 \leq (n - 1) - \tau(\overline{G - v}) + 1 \leq n - \tau(G)$.  

Finally, suppose some vertices of $A$ are isolated in $G - v$, and let $B\subseteq A$ be all such vertices.  We claim that $\gamma_0(A - B, \overline{G - v}) \geq \gamma_0(A, G) - 1$  (and thus $\tau(\overline{G - v}) \geq \tau(G) - 1$).  Indeed, suppose this were not the case.  Then there would be some subset $Y$ of the vertices of $\overline{G - v}$ such that $|Y| < \gamma_0(A, G) - 1$ and $A - B \subseteq N(Y)$.  But then, since the vertices of $B$ are isolated by removing $v$ from $G$, we have $B \subseteq N(v)$, and so $A \subseteq N(Y \cup v)$, contradicting our choice of $X$ since $|Y \cup v| < |X|$.   Thus, 
\[
\pd(G) \leq \pd(\overline{G - v}) + 1 \leq (n - |B| - 1) - \tau(\overline{G - v}) - 1 \leq (n - 2) - (\tau(G) - 1) + 1) = n - \tau(G).
\]
\end{proof}

Recall that $X \subseteq V(G)$ is a \emph{vertex cover} of $G$ if every edge of $G$ contains at least one vertex of $X$.  As demonstrated by Theorem \ref{edgedom}, dominating sets naturally arise in the study of the projective dimensions of edge ideals.  Before going further, we need two well-known observations and a standard proposition, whose short proofs we include for completeness.  

\begin{observation}\label{bh}
The maximum size of a minimal vertex cover of $G$ equals $\bh(I(G))$.
\end{observation}

\begin{proof}
Let $P$ be a minimal vertex cover with maximal cardinality. Then $P$ is an associated prime of $S/I$, so 
\[
\pd(S/I) \geq \pd_{S_P}(S/I)_P = \dim S_P = \height(P).
\]
\end{proof}

\begin{observation}\label{vertexcover}
A subset $X \subseteq V(G)$ is a vertex cover if and only if $V(G)\setminus X$ is independent.  Moreover, $V(G)\setminus X$ is dominating if and only if $X$ is minimal.
\end{observation}

\begin{proof}
The first claim is immediate, since $G[V(G)\setminus X]$ contains no edges if and only if every edge of $G$ contains some vertex of $X$.  For the second claim, suppose $X$ is not minimal, meaning $X - v$ is a vertex cover for some $v \in X$.  Then $V(G) \setminus (X - v)$ would be independent, meaning $N(v) \subseteq X$.  But then $\st(v) \subseteq X$, and $V(G) \setminus X$ is not dominating.  Reversing this argument proves the converse.  
\end{proof}

\begin{proposition}\label{pdlowbound}
Let $G$ be a graph on $n$ vertices.  Then 
\[
\pd(G) \geq n - i(G).
\]
\end{proposition}

\begin{proof}
Since $i(G)$ is the smallest size of an independent dominating set, $n - i(G)$ is the maximum size of a minimal vertex cover, meaning $n - i(G) = \bh(I(G))$, by Observation \ref{bh}.  Since $\pd(I(G)) \geq \bh(I(G))$, Observation \ref{vertexcover} applies.
\end{proof}

In \cite{alh}, the authors show the following.

\begin{theorem}[\cite{alh}]
Let $G$ be a graph on $n$ vertices.  Then $i(G) + \gamma_0(G) \leq n$.
\end{theorem}

Proposition \ref{pdlowbound} then yields the next corollary.

\begin{corollary}
For any graph $G$ without isolated vertices, we have
\[
\pd(G) \geq \gamma_0(G).
\]
\end{corollary}

\section{Chordal and Long Graphs}\label{chordalsection}

We prove in this section strong bounds and even an exact formula on projective dimension for any hereditary class of graphs satisfying a certain dominating set condition.  We then observe that chordal graphs satisfy this property, obtaining an equality (Corollary \ref{chordalcorollary}) for $\pd(G)$ when $G$ is chordal. We also show a result of independent interest: for a graph such that all vertices of degree at least three are independent, the domination and independent domination number coincide.  Some of our results  were inspired by those in \cite{abz} and \cite{kimura}.

\begin{theorem}\label{largegamma}
Let $f(G)$ be either $\gamma(G)$ or $i(G)$. Suppose $\mathcal{G}$ is a hereditary class of graphs such that whenever $G \in \mathcal{G}$ has at least one edge, there is some vertex $v$ of $G$ with $f(G) \leq f(G - v)$.  Then for any $G \in \mathcal{G}$ with $n$ vertices,
\[
\pd(G) \leq n - \gamma(G).
\]
If $f(G)$ is $i(G)$, then for any $G \in \mathcal G$ with $n$ vertices, we have the equality
\[
\pd(G) = n - i(G) = \bh(I(G)), 
\]
and $S/I(G)$ is Cohen-Macaulay if and only if it is unmixed. 
\end{theorem}

\begin{proof}
We just need to check the condition of Theorem \ref{ME} is satisfied for $f(G)$. Set $k = 1$ and let $v_1$ be a vertex of $G$ with $f(G) \leq f(G - v_1)$. Note that $f(G) \leq f(G - \st v_1) + 1$, since  $X \cup \{v_1\}$ is a (independent) dominating set of $G$ whenever $X$ is a (independent) dominating set of $G - \st v_1$.  Writing $G_1$ for $G - v_1$ (as in Theorem \ref{ME}), note that by definition $\is(G_1)$ is contained in any dominating  or independent dominating set of $G_1$, meaning $f(\overline{G_1}) + \is(G_1) = f(G_1)$.  So, we have
\[
f(\overline{G_1}) + \is(G_1) = f(G_1) = f(G - v_1) \geq f(G).
\]
Thus, $\pd(G) \leq n - f(G)$.  In the case when $f(G) = i(G)$, Proposition \ref{pdlowbound} gives equality.  For the last claim, note that $S/I(G)$ is unmixed if and only if $\text{height}(I(G)) = \bh(I(G))$. 
\end{proof}

\begin{remark}
{\rm
We note that the contrapositive of Theorem \ref{largegamma} may prove interesting.  Indeed, if an $n$-vertex graph $G$ fails to satisfy the bound $\pd(G) \leq n - f(G)$ (where $f(G)$ is either $\gamma(G)$ or $i(G)$), then $G$ must contain an induced subgraph $G'$ such that $f(G' - v) = f(G) - 1$ for all $v\in V(G')$ (it is easy to see that $f$ can decrease by at most $1$ upon the removal of a vertex).  That is, $G'$ is a so-called \emph{domination-critical graph}.  Such graphs have been studied at length; see \cite{domcrit} and references given therein.
}
\end{remark}

\begin{corollary}
Let $\mathcal{G}$ be as in Theorem \ref{largegamma}, choose $G \in \mathcal{G}$ with $n$ vertices, and let $d$ be the greatest degree occurring in $G$.  Then 
\[
\pd(G) \leq n\left( 1 - \frac{1}{d+1}\right).
\]
\end{corollary}

\begin{proof}
For any dominating set $A \subseteq V(G)$, note that at most $d|A|$ vertices are adjacent to vertices of $A$.  Since every vertex in $V \setminus A$ is adjacent to some vertex of $A$, it follows that $G$ has at most $d|A| + |A| = |A|(d+1)$ vertices.  Thus, we can write $n \leq (d+1)\gamma(G) \Rightarrow \gamma(G) \geq \frac{n}{d+1}$.  By Theorem \ref{largegamma}, 
\[
\pd(G) \leq n - \gamma(G) \leq n - \frac{n}{d+1} = n\left(1 - \frac{1}{d+1}\right).
\]
\end{proof}


\begin{lemma}\label{delete}
Suppose $N(v) -w \subseteq N(w) $ and $(v,w)$ is an edge.  Then $\gamma(G) \leq \gamma(G - w)$ and $i(G) \leq i(G-w)$. 
\end{lemma}

\begin{proof}
Let $X$ be a (independent) dominating set of $G - w$; It suffices to show that $X$ dominates $G$.  Since $v \in G - w$, $X$ either contains $v$ or a neighbor of $v$.  In the first case, since $(v, w)$ is an edge of $G$, $X$ still dominates $G$.  For the second case, note that any neighbor of $v$ is a neighbor of $w$ (since $N(v) - w \subseteq N(w) $), and thus $X$ must be a (independent) dominating set of $G$ as well. 
\end{proof}

\begin{theorem}[Dirac]\label{dirac}
Let $G$ be a chordal graph with at least one edge.  Then there exists a vertex $v$ of $G$ so that $N(v) \neq \emptyset $ and $G[N(v)]$ is complete.  
\end{theorem}


Using Dirac's Theorem and Theorem \ref{largegamma}, we can recover a formula for the projective dimension of a chordal graph.

\begin{corollary}\label{chordalcorollary}
The class of  chordal graphs satisfies the conditions of Theorem  \ref{largegamma}, and so 
\[
\pd(G) = |V(G)| - i(G).  
\] 
for any chordal graph $G$.
\end{corollary}

\begin{proof}
Let $v$ be as in Theorem \ref{dirac}, and let $w$ be any neighbor of $v$.  Since $G[N(v)]$ is complete, we have $N(v)- w \subseteq N(w) $, and so Lemma \ref{delete} applies. 
\end{proof}

\begin{remark}
{\rm
In fact, the above formula for projective dimension holds for all \emph{sequentially Cohen-Macaulay} graphs; this follows from Smith's results on Cohen-Macaulay complexes (\cite{desmith}) and a theorem of Francisco and Van Tuyl (\cite{fv}) which shows that chordal graphs are sequentially Cohen-Macaulay (\cite{fv}).  For details, see  \cite[Theorem 3.33]{mv} (we thank Russ Woodroofe and Rafael Villarreal for pointing  out the correct result and references). However, Theorem  \ref{largegamma} can be applied to graphs that are not sequentially Cohen-Macaulay, see the next Remark. 
}
\end{remark}

\begin{remark}\label{rmsCM}
{\rm
We note that there are hereditary classes of graphs satisfying the hypotheses of Theorem \ref{largegamma} which properly contain the class of chordal graphs but are not contained in the class of sequentially Cohen-Macaulay graphs.  For instance, let $C_n$ denote the $n$-vertex cycle.  If $n$ is congruent to $0$ or $2 \mod 3$, then $\gamma(C_n) = \gamma(C_n - v)$ and $i(C_n) = i(C_n -v)$ for any $v\in V(C_n)$; see Example \ref{tab}.  Since $C_n - v$ is a tree (and therefore chordal), the following hereditary class satisfies the hypotheses of Theorem \ref{largegamma}: 
\[
\mathcal{G} = \{G : G \text{ is chordal or } G = C_n \text{ for some } n \equiv 0, 2 \text{ mod } 3\}.
\]
By \cite{fv}, $C_n$ is not sequentially Cohen-Macaulay for $n \neq 3,5$.  

More generally, let $\mathcal{G}$ be a hereditary class of the type in Theorem \ref{largegamma}, and let $X$ be a set of graphs $G$ satisfying the following: for any $G \in X$, there is some $v\in V(G)$ such that $f(G) \leq f(G - v)$, and $G - v \in \mathcal{G}$.  Then $\mathcal{G} \cup X$ again satisfies the hypotheses of Theorem \ref{largegamma}.  
}
\end{remark}

\begin{remark}\label{c4}
{\rm In looking for possible generalizations of Corollary \ref{chordalcorollary}, one may be tempted to ask if the same bound holds for perfect graphs (as all chordal graphs are perfect).  However, this is easily seen to be false for the $4$-cycle $C_4$, as $\pd(C_4) = 3$, but $n - i(G) = 4 - 2 = 2$.}
\end{remark}

\begin{corollary}\label{path}
Let $P_n$ denote the path on $n$ vertices.  Then 
\[
i(P_n) =  \left\lceil \frac{n}{3} \right\rceil \text{ and so } \pd(P_n) = \left\lfloor \frac{2n}{3} \right\rfloor.
\]
\end{corollary}

Given Theorem \ref{largegamma}, it makes sense to ask when a graph $G$ satisfies $i(G) = \gamma(G)$.  This has been shown to hold for claw-free graphs by Allan and Laskar (\cite{al}).  Below, we define a fairly expansive class of graphs, called \emph{long} graphs, for which this property holds. 

\begin{definition}
We call a graph $G$ on vertex set $V$ \emph{long} if the set $\{x \in V : \deg(x) > 2\}$ is independent. 
\end{definition}

\begin{theorem}\label{long}
Let $G$ be a long graph.  Then $i(G) = \gamma(G)$.  
\end{theorem}

\begin{proof}
Let $G'$ denote the graph $G$ with all vertices of degree $>2$ deleted.  Then $G'$ is a disjoint union of paths.  Let $P$ be a maximal path in $G'$, and note that no non-endpoint of $P$ can be adjacent to a vertex of degree $>2$ in $G$ (since then its degree would exceed $2$).  Similarly, note that each endpoint of $P$ can be connected (in $G$) to at most one vertex of degree $>2$.  We write $\overline{P}$ to denote the path $P$, plus any vertices of degree $>2$ which are adjacent to the endpoints of $P$.  

Now let $A$ be a dominating set of vertices of $G$.  We construct a dominating set $B$ with $|B| \leq |A|$, such that $B \cap \overline{P}$ is independent for all maximal paths $P$ of $G'$.  Since no two vertices of degree $>2$ are neighbors, it will follow that $B$ is independent.  By further assuming $A$ to be of minimal cardinality, this will prove the theorem.    

If $\overline{P} = P$, we can clearly replace $A \cap P$ with an independent dominating set whose size does not exceed $|A \cap P|$ (by Corollary \ref{path}).  

Next, order the remaining maximal paths of $G': P_1, P_2, \ldots, P_k$.  Fix some $P_i$, and set $P = P_i$.  Let $P$ have vertex set $x_1, x_2, \ldots, x_t$ and edge set $(x_j, x_{j+ 1 \text{ mod }  t})$.  Then $\overline{P} - P$ consists of either one or two vertices.  

First, suppose that $\overline{P} - P$ consists of one vertex, and that it is adjacent to $x_1$.  Call this vertex $v$.  If $v \in A$, replace $A \cap P$ with an independent dominating set of $P - x_1$.  If $v\notin A$ and no neighbor of $v$ not in $P$ is contained in $A$, we must have $x_1 \in A$.  Thus, we replace $A \cap (P - x_1)$ with an independent dominating set of $P- x_1 - x_2$.  Finally, if $v \notin A$ but some other path contains a neighbor of $v$ that is in $A$, we replace $A \cap P$ with an independent dominating set of $P$.  The case when $v$ is adjacent to both $x_1$ and $x_t$ is handled similarly: if $v$ is not adjacent to any other element of $A$, then replace $A \cap (P- x_1)$ with a dominating independent set of $P - x_1 - x_2$ (and keep $x_1 \in A$).  

Now, we focus on the case when $\overline{P}-P$ consists of two vertices.  Let $v$ be the neighbor of $x_1$, and let $y$ be the neighbor of $x_2$.  If $v, y \in A$, replace $A \cap P$ with an independent dominating set of $P - x_1 - x_t$.  If, without loss of generality, $v \in A$ and $y\notin A$, there are two subcases.  First, suppose that $x_t$ is the only neighbor of $y$ contained in $A$.  Then we may replace $A \cap (P - x_t)$ with an independent dominating set of $P - x_1 - x_t - x_{t-1}$.  If on the other hand $A$ contains a neighbor of $y$ in another path, we replace $A \cap P$ with an independent dominating set of $P$. 

Next, consider the case where $v, y \notin A$.  If $x_1$ is the only neighbor of $v$ contained in $A$ and $x_t$ is the only neighbor of $y$ contained in $A$, then either $t \geq 3$, and we may replace $A \cap (P - x_1 - x_t)$ with an independent dominating set of $P - x_1 - x_t - x_2 - x_{t-1}$, or $x_1$ and $x_t$ are neighbors (that is, $t = 2$).  In the latter case, we simply replace $x_1$ and $x_2$ with $v$ and $y$.  If, without loss, $x_1$ is the only neighbor of $v$ in $A$ but $y$ has neighbors other than $x_t$ in $A$, we may replace $A\cap (P - x_1)$ with an independent dominating set of $P- x_1 - x_2$.  Finally, if both $v$ and $y$ have neighbors in $A - P$, we may replace $A \cap P$ with an independent dominating set of $P$.  

Performing this process on each path yields a dominating set $B$ such that no two vertices $x, x' \in B \cap \overline{P}$ (for any path $P$) are neighbors.  It follows, then, that $B$ is independent.  
\end{proof}

It is worth noting that the converse to Theorem \ref{long} is false, as shown by Figure \ref{hgraph}.\\

\begin{figure}[htp]
\centering
\includegraphics[height = .9in]{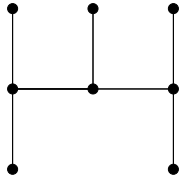}
\caption{A tree $T$ with $i(T) = \gamma(T) = 3$.}\label{hgraph}
\end{figure}

\begin{corollary}
Let $G$ be a graph, and let $G'$ be the graph obtained by subdividing each edge $(v, w)$ of $G$ whenever both $\deg(v)$ and $\deg(w) $ exceed $2$.  Then $i(G') = \gamma(G')$.
\end{corollary}

\begin{proof}
It is easily seen that $G'$ is a long graph.
\end{proof}

It is not always true that $\pd(G) = |V(G)| - i(G)$ for a long graph $G$ (see Remark \ref{c4}). However, combining Theorem \ref{largegamma} and \ref{long} we have:

\begin{corollary}
Suppose $\mathcal{G}$ is a hereditary class of long graphs such that whenever $G \in \mathcal{G}$ has at least one edge, there is some vertex $v$ of $G$ with $\gamma(G) \leq \gamma(G - v)$.  Then for any $G \in \mathcal{G}$ with $n$ vertices,
\[
\pd(G) = n - \gamma(G) = n -i(G).
\]

\end{corollary}

\section{Homology and Connectivity of Independence Complexes}\label{homsec}
In this section we collect various corollaries on connectivity of independence complexes. All the hard work has been done in the previous Sections. For the definitions of various domination parameters, see Section \ref{dominationsection}.

As Hochster's formula (Corollary \ref{hochster}) limits the possible non-zero homology of induced subcomplexes, the numerous bounds obtained earlier for the projective dimension of a graph also allow us to detect vanishing homology.  Indeed, if $G$ is a graph and we set $W = V(G)$ in Proposition \ref{hochster}, we have the following.

\begin{corollary}\label{pdhomology}
Let $G$ be a graph on $n$ vertices.  Then $\tilde{H}_{k}(\ind(G)) = 0$ whenever $k  < n - \pd(G) - 1$.
\end{corollary}

\begin{remark}
{\rm Let $\Delta$ be a simplicial complex.  By The Hurewicz Isomorphism (see, for instance, \cite{hatcher}), $\Delta$ is $m$-connected if and only if $\tilde{H}_k(\Delta) = 0$ for all $k \leq m$ and $\pi_1(\Delta) = 0$.  Thus the results in this section also show homotopic connectivity when the independence complex in question is simply connected.  Many graphs are known to have simply connected independence complexes, such as graphs $G$ with $\gamma_0(G) > 4$ (\cite{st}) and \emph{most} claw-free graphs (\cite{engstromclawfree}), thus our results give bounds on the homotopic connectivity in those cases.  In general, there is no algorithm to determine if a given independence complex is simply connected (see \cite{abzdisproof}), but see \cite{as} and \cite{rs} for approaches that work well in practice.}
\end{remark}

In \cite{abz}, Aharoni, Berger, and Ziv show the following.

\begin{theorem}[\cite{abz}]
If $G$ is chordal, $\tilde{H}_k(\ind(G)) = 0$ for $k < \gamma(G) - 1$.
\end{theorem}

In \cite{russ}, Woodroofe proves the following strengthening of this result (chordal graphs are sequentially Cohen-Macaulay; see \cite{fv}).

\begin{theorem}[\cite{russ}]
Let $G$ be a sequentially Cohen-Macaulay graph.  Then $\tilde{H}_k(\ind(G)) = 0$ for $k < i(G) -1$. 
\end{theorem}

As a first application of Corollary \ref{pdhomology}, we can recover this result in the case when $G$ is chordal. See also Theorem \ref{largegamma} and Remark \ref{rmsCM}.

\begin{corollary}
If $G$ is chordal, $\tilde{H}_k(\ind(G)) = 0$ for $k < i(G) -1$.  
\end{corollary}

\begin{proof}
If $k < n - \pd(G) - 1$, Corollary \ref{chordalcorollary} gives us that $k < i(G) - 1$, and we can apply Corollary \ref{pdhomology}.
\end{proof}

Similarly, we obtain the following general bound for the homology of a graph independence complex by using Theorem \ref{edgedom}.

\begin{corollary}\label{edgedomhomology}
Let $G$ be a graph.  Then $\tilde{H}_k(\ind(G)) = 0$ whenever $k < \epsilon(G) - 1$.
\end{corollary}

Corollary \ref{edgedomhomology} also allows us to recover a related result of Chudnovsky.  

\begin{corollary}[\cite{chudnovsky}]
Let $G$ be a graph.  Then $\tilde{H}_k(\ind(G)) = 0$ whenever $k < \frac{\gamma_0(G)}{2} - 1$.
\end{corollary}

\begin{proof}
 The result follows from Proposition \ref{gammaind} and Corollary \ref{edgedomhomology}. 
\end{proof}

Theorem \ref{pdind}  allows us to prove the following, originally shown in \cite{ah} (albeit with different terminology).  

\begin{corollary}\label{tauhomology}
Let $G$ be a graph.  Then $\tilde{H}_k(\ind(G)) = 0$ for $k < \tau(G) -1$. 
\end{corollary}

In \cite{barmak}, Barmak proves the following two theorems.

\begin{theorem}[\cite{barmak}]\label{firstbarmak}
Let $G$ be a claw-free graph.  Then $\ind(G)$ is $\left\lceil \frac{\dim(\ind(G) -3}{2} \right\rceil$-connected.
\end{theorem}

\begin{theorem}[\cite{barmak}]\label{secondbarmak}
Let $G$ be a graph with $A \subseteq V(G)$, and suppose the distance between any two vertices of $A$ is at least $3$.  Then $G$ is $(|A| - 2)$-connected.  
\end{theorem}

The following Corollary shows that the homological analogue of Theorem \ref{firstbarmak} is actually a special case of a more general phenomenon for graphs which are $K_{1, m}$-free. 

\begin{corollary}
Let $G$ be a $K_{1, m}$-free graph without isolated vertices.  Then $\tilde{H}_k(\ind(G)) = 0$ for 
\[
k \leq \left\lceil \frac{\dim(\ind(G)) - 2m + 3}{m-1} \right\rceil.
\]
\end{corollary}

\begin{proof}
Let $A$ be an independent set of vertices of $G$ with maximal cardinality (so that $|A| = \dim(\ind(G)) + 1$), and let $X$ be a set dominating $A$ with $|X| = \gamma_0(A, G) \leq \tau(G)$.  If $x \in X$, the number of elements in $A$ that are neighbors of $x$ cannot exceed $m-1$.  Indeed, if $x$ had $m$ neighbors in $A$, then these vertices together with $x$ would form an induced $K_{1, m}$ (since $A$ is independent).  Thus, we have $|A| \leq |X| (m-1) \Rightarrow \tau(G) \geq \frac{|A|}{m-1}$, meaning $\tau(G) - 1 \geq \left\lceil \frac{\dim(\ind(G))  - m +2}{m-1}\right \rceil$.  The result now follows from Corollary \ref{tauhomology}.
\end{proof}

We also note that the homological version of Theorem \ref{secondbarmak} follows easily from our results.

\begin{corollary}
Let $G$ be a connected graph and let $A\subseteq V(G)$ such that the distance between any two members of $A$ is at least $3$.  Then $\tilde{H}_k(\ind(G)) = 0 $ for $k \leq |A| - 2$.  
\end{corollary}

\begin{proof}
The set $A$ is independent since none of its members are distance $1$ from each other.  Now let $X$ be a set of vertices realizing $\gamma_0(A, G)$.  If $x \in X$ is adjacent to two vertices in $A$, then these two vertices would be distance $2$ from one another.  Thus, $\tau(G) \geq |X| = |A|$, and so Theorem \ref{pdind} completes the proof.    
\end{proof}

We close this section with the homological corollaries of Theorems \ref{mcondition} and \ref{zell}, each of which follows immediately from the application of \ref{pdhomology}.  These results improve and generalize the bounds on the homology of claw-free graphs (\cite{engstromclawfree}) and subgraphs of the lattice $\mathbb{Z}^2$ (\cite{engstromtree}). 

\begin{corollary}
Let $G$ be a graph containing no induced $K_{1, m+1}$, and let $(x, y)$ be an $\frac{m-1}{m}$-max edge with $\deg(x) = d \geq e = \deg(y)$.  Then $\tilde{H}_k(\ind(G)) = 0$ for 
\[
k < \frac{|V(G)|}{d+ \frac{m-1}{m}e + 1} - 1.
\]
\end{corollary}

\begin{corollary}
Let $G$ be a subgraph of the lattice $\mathbb{Z}^\ell$.  Then $\tilde{H}_k(\ind(G)) = 0$ for $k < \frac{|V(G)|}{2\ell +1} - 1$.
\end{corollary}

\section{Further Remarks}\label{lastsec}

In this section we discuss some potential research directions from our work and give some relevant examples. The most natural next step is  to give bounds on projective dimension of monomial ideals in general. We expect some of our results to extend smoothly with the right notions of various domination parameters for clutters. However, significant progress in this direction may require deeper insight; see the concluding discussion of \cite{craigandus}. 

Another interesting question is whether the process used in Theorem \ref{ME} can be used to compute the projective dimension. 

\begin{question}\label{abz} 
Let $\mathcal{C}$ be a class of graphs whose projective dimensions are independent of the chosen ground field, and let $F(G)$ be the maximal of functions satisfying the conditions of Theorem \ref{ME} for graphs belonging to $\mathcal{C}$.  Is it then the case that $\pd(G) = |V(G)| - F(G)$ for any graph $G\in \mathcal{C}$?
\end{question}

Note that a  similar conjecture on connectivity of independence complexes was  raised by Aharoni-Berger-Ziv in \cite{abzconj} and disproved by Barmak in \cite{abzdisproof}. However, since projective dimension is clearly computable, it is not clear to us that one can answer our question negatively in the same manner.

Given our two main theorems bounding projective dimension (Theorems \ref{pdind} and \ref{edgedom}), it makes sense to juxtapose the two statistics $\epsilon(G)$ and $\tau(G)$.  

\begin{example}\label{eptau}
{\rm We note that neither of Theorems \ref{pdind} nor \ref{edgedom} implies the other.  To see this, we construct two infinite families of connected graphs: For a graph $G$ in the first family, $\epsilon(G)$ is roughly twice $\tau(G)$, whereas the opposite is true for graphs in the second family.  It should be noted that it is not hard to come up with disconnected examples of this phenomenon: Indeed, note that $\epsilon(C_5) = 2$ and $\tau(C_5) = 1$, so that if we let $G$ be the disjoint union of $k$ copies of $C_5$ we have $\epsilon(C_5) = 2k$ and $\tau(G) = k$.  Similarly, note that $\epsilon(P_4) = 1$ and $\tau(P_4) = 2$, so that if we let $G$ be the disjoint union of $k$ copies of $P_4$ we have $\epsilon(G) = k$ and $\tau(G) = 2k$.   

Let $Q_n$ denote the graph consisting of $n$ copies of $C_5$, connected in series with an edge between each copy (see Figure \ref{pent}).  Then clearly $\epsilon(Q_n) = 2n-1$ (choose each of the edges connecting the pentagons, plus the bottom edge from each pentagon).  Furthermore, any independent set in $Q_n$ contains at most $2$ vertices from each pentagon, and these two vertices can be dominated by one vertex, giving $\tau(Q_n) = n$.  

\begin{figure}[htp]
\centering
\includegraphics[height = .75in]{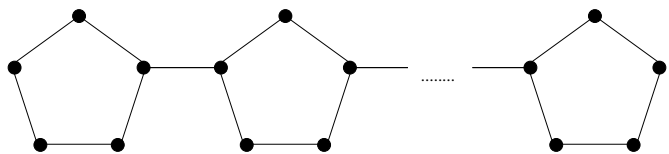}
\caption{The graph $Q_n$.}\label{pent}
\end{figure}

On the other hand, let $T_n$ denote the path on $4n$ vertices $v_1, v_2, \ldots, v_{4n}$, with additional vertices $w_1, w_2, \ldots, w_n$ where, for each $i$, $w_i$ is connected to $v_{4(i-1)+ 1}$ (see Figure \ref{t2}).  Then $\epsilon(T_n) = n$ (simply take all edges of the form $(v_{4k + 1}, v_{4k+2})$ for $0 \leq k \leq n-1$), yet $\tau(T_n) \geq 2n$, which can be seen by considering the independent set $\{v_4, v_8, v_{16}, \ldots, v_{4n}\} \cup \{w_1, w_2, \ldots, w_n\}$.  }

\begin{figure}[htp]
\centering
\includegraphics[height = .55in]{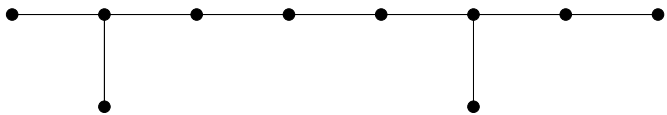}
\caption{The graph $T_2$.}\label{t2}
\end{figure}

\end{example}

\begin{example}\label{tab}
{\rm Let $P_n$ denote the path on $n$ vertices, and let $C_n$ denote the $n$-vertex cycle.  Since both $P_n$ and $C_n$ are long graphs, Theorem \ref{long} gives us that $\gamma(G) = i(G)$ for $G$ a cycle or path.  In the table below we give the parameters for cycles and paths.}  

\begin{center}
\begin{tabular}{|c|c|c|c|c|}
\hline
\text{Graph} $G$ & $i(G) = \gamma(G)$ & $\epsilon(G)$ & $\tau(G)$  & $\pd(G)$  \\
\hline
$P_n$ & $\lceil \frac{n}{3} \rceil$ &  $\lceil \frac{n}{4} \rceil$ & $\lceil \frac{n}{3} \rceil$ &        $\lfloor \frac{2n}{3} \rfloor $\\
\hline
$C_n$ & $\lceil \frac{n}{3} \rceil$  & $\lceil \frac{n}{4} \rceil$ & $\lfloor \frac{n}{3}  \rfloor$          &$\lceil \frac{2n-1}{3} \rceil$\\
\hline
\end{tabular}
\end{center}
\end{example}

\begin{remark}
{\rm Given Theorem \ref{largegamma}, it makes sense to ask when the bound $\pd(G) \leq n -\gamma(G)$ holds for a graph $G$ with $n$ vertices.  Note that when $G$ is long, $\pd(G) \neq n - i(G)$ implies $\pd(G) \nleq n - \gamma(G)$ (since $i(G) = \gamma(G) $ and $\pd(G) \geq n - i(G)$).  Thus, it is not hard to produce graphs for which the bound $\pd(G) \leq n - \gamma(G)$ fails.  As $\gamma(G) \geq \tau(G)$ (Proposition \ref{gammaind}), it is tempting to ask if there is a combinatorial graph invariant (call it $\kappa$) such that $\tau(G) \leq \kappa(G) \leq \gamma(G)$ for all $G$, $\kappa(G) = i(G)$ when $G$ is chordal, and $\pd(G) \leq n - \kappa(G)$ for all $n$-vertex graphs $G$.}
\end{remark}  

\noindent \textbf{Acknowledgement}:  We are grateful to Craig Huneke and Jeremy Martin for many interesting conversations.  The first author was partially supported by NSF grant DMS 0834050.


\begin{thebibliography}{BroSh}

\bibitem{abz} R. Aharoni, E. Berger, and R. Ziv, A tree version of K\"onig's Theorem.  \emph{Combinatorica} 22 (2002), 335--343.  

\bibitem{abzconj} R. Aharoni, E. Berger, and R. Ziv, Independent systems of representatives in weighted graphs.  \emph{Combinatorica} 27 (2007), 253--267.  

\bibitem{ah} R. Aharoni and P. Haxell, Hall's theorem for hypergraphs.  \emph{J. Graph Theory} 35 (2000), 83--88.

\bibitem{al} R. Allan and R. Laskar, On domination and independent domination numbers of a graph.  \emph{Discrete Math.} 23 (1978), no. 2, 73--76.

\bibitem{alh} R. Allan, R. Laskar, and S. Hedetniemi, A note on total domination. \emph{Discrete Math.} 49 (1984), no. 1, 7--13.

\bibitem{as} M. Aschbacher and Y. Segev, Extending morphisms of groups and graphs. \emph{Ann. of Math.} (2) 135 (1992), no. 2, 297--323.

\bibitem{barmak} J. Barmak, Star clusters in independence complexes of graphs.  \emph{preprint} (2010).

\bibitem{abzdisproof} J. Barmak, The word problem and the Aharoni-Berger-Ziv conjecture on the connectivity of independence complexes.  \emph{preprint} (2010).  

\bibitem{bollobas} B. Bollob\'as, Graph Theory.  Graduate Texts in Mathematics, 63.  Springer-Verlag, New York (1979). 

\bibitem{chudnovsky} M. Chudnovsky, Systems of disjoint representatives. M.Sc. Thesis.  Technion, Haifa, Israel (2000). 

\bibitem{craigandus} H. Dao, C. Huneke, and J. Schweig, Bounds on the regularity and projective dimension of ideals associated to graphs, \emph{preprint} (2011).

\bibitem{DE} A. Dochtermann and A. Engstr\"om, Algebraic properties of edge ideals via combinatorial topology.  \emph{Electron. J. Combin.} 16 (2009),
 Special volume in honor of Anders Bjorner, Research Paper 2.



\bibitem{engstromclawfree} A. Engstr\"om, Independence complexes of claw-free graphs.  \emph{European J. Combin.} 29 (2008), no. 1, 234--241.

\bibitem{engstromtree} A. Engstr\"om, Complex of directed trees and independence complexes.  \emph{Discrete Math.} 29 (2009), no. 10, 3299--3309.


\bibitem{fv} C. Francisco and A. Van Tuyl, Sequentially Cohen-Macaulay edge ideals.  \emph{Proc. Amer. Math. Soc.} 135 (2007), no. 8, 2327--2337.

\bibitem{hv} T. H\`a and A. Van Tuyl, Monomial ideals, edge ideals of hypergraphs, and their graded Betti numbers.  \emph{J. Algebraic Combin.} 27 (2008), no. 2, 215--245.

\bibitem{hatcher} A. Hatcher, Algebraic Topology. Cambridge University Press, Cambridge (2002). 

\bibitem{domcrit} T. Haynes and M. Henning, Changing and unchanging domination: a classification.  In honor of Frank Harary.  \emph{Discrete Math.} 272 (2003), no. 1, 65--79.

\bibitem{dombook} T. Haynes, S. Hedetniemi, and P. Slater, Fundamentals of domination in graphs. Monographs and Textbooks in Pure and Applied Mathematics, 208.  Marcel Dekker, Inc., New York (1998). 


\bibitem{kimura} K. Kimura, Non-vanishingness of Betti numbers of edge ideals, \emph{preprint} (2011) arXiv:1110.2333v2

\bibitem{manoj} M. Kummini, Homological invariants of monomial and binomial ideals, thesis, University of Kansas (2008).

\bibitem{mesh} R. Meshulam, Domination numbers and homology. \emph{J. Combin. Theory Ser. A} 102 (2003), no. 2, 321--330.

\bibitem{millersturmfels} E. Miller and B. Sturmfels, Combinatorial Commutative Algebra. Graduate Texts in Mathematics, 227.  Springer-Verlag, New York (2005). 

\bibitem{mv} S. Morey and R. Villarreal, Edge ideals: algebraic and combinatorial properties, \emph{preprint} (2010) arXiv:1012.5329v3.

\bibitem{rs} S. Rees and L. Soicher, An algorithmic approach to fundamental groups and covers of combinatorial cell complexes. \emph{J. Symbolic Comput.} 29 (2000), no. 1, 59--77.

\bibitem{desmith} D. Smith, On the Cohen-Macaulay property in commutative algebra and simplicial topology. 
\emph{Pacific J. Math.} 141 (1990), no. 1, 165--196. 

\bibitem{st} T. Szab\'o and G. Tardos, Extremal problems for transversals in graphs with bounded degree. \emph{Combinatorica} 26 (2006), no. 3, 333--351. 

\bibitem{russ} R. Woodroofe, Vertex decomposable graphs and obstructions to shellability, \emph{Proc. Amer. Math. Soc.} 137 (2009), no. 10, 1218--1227.

\end{thebibliography}
\end{document}